\documentclass[12pt]{article}

\usepackage[utf8]{inputenc}
\usepackage[english]{babel}

\usepackage[T1]{fontenc}

\usepackage{tikz}
\usepackage{verbatim}

\usepackage{amsmath,amsfonts,amssymb,amsthm}
\usepackage{nccmath}
\usepackage{mathrsfs,indentfirst, latexsym}
\usepackage{amscd}
\usepackage{eucal}
\usepackage{fancyhdr}
\numberwithin{equation}{section}

\newtheorem{theorem}{Theorem}[section]
\newtheorem{lemma}[theorem]{Lemma}
\newtheorem{proposition}[theorem]{Proposition}

\newtheorem{question}[theorem]{Question}

\theoremstyle{definition}
\newtheorem{definition}[theorem]{Definition}

\newtheorem{remark}[theorem]{Remark}

\newcommand{\N}{\mathbb{N}} 
\newcommand{\Z}{\mathbb{Z}} 
\newcommand{\R}{\mathbb{R}} 
\newcommand{\C}{\mathbb{C}} 

\newcommand{\PS}{\mathcal{PS}}
\newcommand{\Synd}{\mathcal{S}}
\newcommand{\IP}{\mathcal{IP}}

\newcommand{\E}{\mathcal{D}}
\newcommand{\I}{\mathcal{I}}

\newcommand{\F}{\mathscr{F}}

\newcommand{\Lin}{\mathcal{L}}

\newcommand{\abs}[1]{\left\lvert#1\right\rvert}

\newcommand{\norm}[1]{\lVert#1\rVert}

\begin{document}

\title{A mixing operator $T$ for which $(T, T^2)$ is not disjoint transitive}

\author{Yunied Puig\footnote{Universit\`{a} degli Studi di Milano, Dipartimento di Matematica "Federigo Enriques", Via Saldini 50 - 20133 Milano, Italy. e-mail:puigdedios@gmail.com}}
\date{}
\maketitle

\begin{abstract}
Using a result from ergodic Ramsey theory, we answer a question posed by B\`{e}s, Martin, Peris and Shkarin by showing a mixing operator $T$ on a Hilbert space such that the tuple $(T, T^2)$ is not disjoint transitive. 
  \end{abstract}
 
 
 
\section{Introduction}
 Let $X$ be a separable topological vector space. Denote the set of bounded linear operators on $X$ by $\Lin(X)$. From now on $T$ is considered in $\Lin(X)$. An operator $T$ is called $\emph{hypercyclic}$ provided that there exists a vector $x\in X$ such that its orbit $\{T^nx: n\geq 0\}$ is dense in $X$ and $x$ is called a \emph{hypercyclic vector} for $T$. Hypercyclic operators are one of the most studied objects in linear dynamics, see \cite{GrPe} and \cite{BaMa} for further information concerning concepts, results and a detailed account on this subject. More generally, a tuple of operators $(T_1, \dots, T_N)$ is said to be \emph{disjoint hypercyclic} (\emph{d-hypercyclic} for short) if
 \[
 \{(T_1^nx, \dots, T_N^nx ): n\in \N\}
 \]
is dense in $X^N$ for some vector $x\in X$.   

   If $X$ is an $F$-space, thanks to Birkhoff's theorem \cite{BaMa}, $T$ is hypercyclic if and only if $T$ is \emph{topologically transitive}, i.e. for every non-empty open sets $U, V$ of $X$, the return set $N(U, V):=\{n\geq 0: T^n(U)\cap V\neq \emptyset\}$ is non-empty. If $N(U, V)$ is cofinite for every non-empty open sets $U$ and $V$ then $T$ is said to be \emph{mixing}.
   
  The notion of disjoint transitivity, a strengthening of transitivity is defined as follows: a tuple of operators $(T_1, \dots, T_N)$ is \emph{disjoint transitive} (\emph{d-transitive} for short), if for any $(N+1)$-tuple $(U_i)_{i=0}^N$ of non-empty open sets 
  \[
  N_{U_1,\dots, U_N; U_0}:=\big\{k\geq 0: T^{-k}_{1}(U_1)\cap \dots \cap T^{-k}_{N}(U_N) \cap U_0\neq\emptyset\big\}
  \]
  is non-empty. In particular, if the set $ N_{U_1,\dots, U_N; U_0}$ happens to be cofinite for any $(N+1)$-tuple $(U_i)_{i=0}^N$ of non-empty open sets, then $(T_1, \dots, T_N)$ is said to be \emph{disjoint mixing} (\emph{d-mixing} for short). Connection between $d$-hypercyclicity and $d$-transitivity can be found in \cite{BePe}.
  
  B\`{e}s, Martin, Peris and Shkarin \cite{BMPS} showed the following: if $T$ is an operator on $X$ satisfying the Original Kitai Criterion, then the tuple $(T, \dots, T^r)$ is $d$-mixing, for any $r\in \N$. As a consequence, any bilateral weighted shift $T$ on $l^p(\Z), (1\leq p<\infty)$ or $c_0(\Z)$ is mixing if and only if $(T, \dots, T^r)$ is $d$-mixing, for any $r\in \N$. Nevertheless, they remarked that this phenomenon does not occur beyond the weighted shift context, by providing an example of a mixing Hilbert space operator $T$ so that $(T, T^2)$ is not $d$-mixing. This result is a partial answer to the following question posed by the authors in the same paper \cite{BMPS}.
   \begin{question}
   \label{questionBMPS}
   Does there exist a mixing continuous linear operator $T$ on a separable Banach space, such that $(T, T^2)$ is not d-transitive?
   \end{question}
   Our aim is to give a positive answer to Question \ref{questionBMPS} (Theorem \ref{TrhcTuplenondTransitive} below).
  
\subsection{Preliminaries and main results}
\label{prelAndMainres}

   Let $A\subseteq \N$, $\abs{A}$ stands for the cardinality of $A$. Let $\F$ be a set of subsets of $\N$, we say that $\F$ is a \emph{family} on $\N$ provided (I.) $ \abs{A}=\infty$ for any $A\in \F$ and (II.) $ A\subset B$ implies $B\in \F$, for any $A\in \F$. From now on $\F$ will be a family on $\N$.

 In a natural way we generalize the notion of disjoint transitivity by introducing what we call \emph{$\F$-disjoint transitivity} (or \emph{$d$-$\F$} for short). 
  \begin{definition}
The tuple of sequence of operators $(T_{1,n_k}, \dots, T_{N, n_k})_k$ is said to be \emph{$d$-$\F$} if for any $(N+1)$-tuple $(U_i)_{i=0}^N$ of non-empty open sets  we have
\[
\big\{k\geq 0: T^{-1}_{1, n_k}(U_1)\cap \dots \cap T^{-1}_{N, n_k}(U_N) \cap U_0\neq\emptyset\big\}\in \F.
\]
In particular, if $T_{i, n_k}=T_i^k$, for any $k\in \N, 1\leq i \leq N$ in the above definition, then the $N$-tuple of operators $(T_1, \dots, T_N)$ is said to be \emph{$d$-$\F$}.  
  \end{definition} 
  
  Observe that in particular whenever $\F$ is the family of non-empty sets or the family of cofinite sets, we obtain the notion of disjoint transitivity and disjoint mixing respectively. On the other hand, if $N=1$ we obtain the $\F$-transitivity notion. More specifically, an operator $T$ is called \emph{$\F$-transitive operator} (or \emph{$\F$-operator} for short) whenever the set $N(U, V):=\{n\geq0: T^n(U)\cap V\neq \emptyset\}$ is in $\F$. This notion was introduced and studied in \cite{BMPP}. 
  
     Recall that an operator $T$ is said to be \emph{chaotic} if it is hypercyclic and has a dense set of periodic points ($x\in X$ is a \emph{periodic point} of $T$ if $T^kx=x$ for some $k\geq 1$). 
     
   An operator $T$ is said to be \emph{reiteratively hypercyclic} if there exists $x\in X$ such that for any non-empty open set $U$ in $X$, the set $N(x, U)=\{n\geq 0: T^nx\in U\}$ has positive upper Banach density, where the upper Banach density of a set $A\subset \N$ is given by \[\overline{Bd}(A)=\lim_{s\to\infty}\frac{\alpha^s}{s},\] and $\alpha^s=\limsup_{k\to \infty} |A\cap [k+1, k+s]|$, for any $s\geq 1$.
Reiteratively hypercyclic operators have been studied in \cite{BMPP1} and \cite{Pu}.
 
  It is known that there exists a reiteratively hypercyclic operator which is not chaotic, see \cite{BaMa}. However, concerning the converse we have the following result due to Menet.
     
     \begin{theorem} Theorem 1.1 \cite{Me}
     \label{theor.Menet}
     Any chaotic operator is reiteratively hypercyclic.
     \end{theorem} 
   Recall that a set $A\subseteq \N$ is \emph{syndetic} set if $A$ has bounded gaps, i.e. if $A$ is enumerated increasingly as $(x_n)_n=A$, then $\max_n x_{n+1}-x_n<M$, for some $M>0$.
  
   In \cite{BMPS}, the authors show that there exists a mixing operator $T$ on a Hilbert space such that $(T, T^2)$ is not $d$-mixing. We show that the same operator satisfies more specific properties.
  
   \begin{theorem}
   \label{op.mixing.tuple.not.d-S}
   There exists $T\in \Lin(l^2)$ such that $T$ is mixing, chaotic  and  $(T, T^2)$ is not $d$-syndetic. 
   \end{theorem}
 So, our result improves the result of \cite{BMPS} already mentioned but still does not answer Question \ref{questionBMPS}. In answering the question, we will see  Szemer\'{e}di's famous theorem unexpectedly playing an important role. Indeed, using a result of ergodic Ramsey theory due to Bergelson and McCutcheon \cite{BeMc}, which is in fact a kind of Szemer\'{e}di's theorem for generalized polynomials we obtain the following result.
 
   \begin{theorem}
   \label{generalcase.rhc.implies.d-synd}
If $T$ is reiteratively hypercyclic then $(T,\dots, T^r)$ is $d$-syndetic or not $d$-transitive, for any $r\in \N$.
   \end{theorem}
   
 Now, by Theorem \ref{theor.Menet}, Theorem \ref{op.mixing.tuple.not.d-S} and Theorem \ref{generalcase.rhc.implies.d-synd} we can deduce our main result which gives a positive answer to Question \ref{questionBMPS}. 
    
\begin{theorem}
\label{TrhcTuplenondTransitive}
There exists a mixing and chaotic operator $T$ in $\Lin(l^2)$ such that $(T, T^2)$ is not d-transitive.
\end{theorem}

  
  \section{Proof of Theorem \ref{TrhcTuplenondTransitive}}
   As already mentioned, in order to prove Theorem \ref{TrhcTuplenondTransitive} it is enough to prove Theorem \ref{op.mixing.tuple.not.d-S} and Theorem \ref{generalcase.rhc.implies.d-synd}.
  \subsection{Proof of Theorem \ref{op.mixing.tuple.not.d-S}}
  
  In Theorem 3.8 \cite{BMPS} the authors show an example of a mixing Hilbert space operator $T$ such that $(T, T^2)$ is not $d$-mixing. We will show that in addition $T$ is chaotic and $(T, T^2)$ is not $d$-syndetic. So in particular it is not $d$-mixing.  We follow exactly the same sketch of proof of Theorem 3.8 \cite{BMPS} introducing minor modifications for our convenience. Nevertheless we describe here all the details for the sake of completeness.
 
  Let $1\leq p < \infty, -\infty< a<b<+\infty$ and $k\in \N$. Recall that the Sobolev space $W^{k, p}[a, b]$ is the space of functions $f\in C^{k-1}[a, b]$ such that $f^{(k-1)}$ is absolutely continuous and $f^{(k)}\in L^p[a, b]$. The space $W^{k, p}[a, b]$ endowed with the norm 
  \[
  \norm{f}_{W^{k, p}[a, b]}=\Big(\int_a^b\Big(\sum_{j=0}^k\abs{f^{(j)}(x)}^p\Big)dx\Big)^{1/p}
  \]
  is a Banach space isomorphic to $L^p[0,1]$. Now, $W^{k, 2}[a, b]$ is a separable infinite-dimensional Hilbert space for each $k\in \N$. The family of operators to be considered lives on separable complex Hilbert spaces and is built from a single operator. Let $M\in \Lin(W^{2, 2}[-\pi, \pi])$ be defined by the formula 
  \begin{equation}
  \label{def.opM}
  M: W^{2, 2}[-\pi, \pi]\longrightarrow W^{2, 2}[-\pi, \pi], \qquad Mf(x)=e^{ix}f(x).
  \end{equation}
  Denote $\mathscr{H}=W^{2, 2}[-\pi, \pi]$ and $M^*$ the dual operator. Then, $M^*\in \Lin(\mathscr{H}^*)$. For each $t\in [-\pi, \pi], \delta_t \in \mathscr{H}^*$, where $\delta_t: \mathscr{H}\longrightarrow \C, \delta_t(f)=f(t)$. Furthermore, the map $t\to \delta_t$ from $[-\pi, \pi]$ to $\mathscr{H}^*$ is norm-continuous. For a non-empty compact subset $K$ of $[-\pi, \pi]$, denote 
  \[
  X_K=\overline{span\{\delta_t: t\in K\}}
  \]
  where the closure of $span\{\delta_t: t\in K\}$ is taken with respect to the norm of $\mathscr{H}^*$.
  
  Now, the functionals $\delta_t$ are linearly independent, $X_K$ is always a separable Hilbert space and $X_K$ is infinite dimensional if and only if $K$ is infinite. The following condition holds 
  \[
  M^*\delta_t=e^{it}\delta_t, \qquad \mbox{for\quad each $t\in [-\pi, \pi]$}. 
  \]
   Hence, each $X_K$ is an invariant subspace for $M^*$, which allows us to consider the operator
   \[
   Q_K\in \Lin(X_K), \qquad Q_K=\left.M^*\right|_{X_K}.
   \]
    The following is taken from \cite{BMPS} and tells us when $Q_K$ is mixing or non-transitive, we omit the proof. 
    \begin{proposition} Proposition 3.9 \cite{BMPS}
    \label{Q_K.tran.or.non-trans}\
    
    Let $K$ be a non-empty compact subset of $[-\pi, \pi]$. If $K$ has no isolated points, then $Q_K$ is mixing. If $K$ has an isolated point, the $Q_K$ is non-transitive. 
    \end{proposition} 
    
    Now, consider the set
    \begin{equation}
    \label{def.set.K}
    K=\Big\{\sum_{n=1}^\infty 2\pi \epsilon_n\cdot \frac{1}{2^{6^n}}: \epsilon\in\{0, 1\}^\N\Big\},
    \end{equation}
    then the operator $Q_K\in\Lin (X_K)$ is the one pointed out in \cite{BMPS} to be mixing such that $(Q_K, Q_K^2)$ is not $d$-mixing.
     In addition, we show that $Q_K$ is chaotic and that $(Q_K, Q_K^2)$ is not $d$-syndetic. This is the content of the next results.
    
   \begin{lemma}
   \label{Q_K_chaotic}
   Let $K$ be the compact subset of $[-\pi, \pi]$ defined in (\ref{def.set.K}), then $Q_K$ is chaotic.
   \end{lemma}
    \begin{proof}
    The operator $Q_K$ is mixing by Proposition \ref{Q_K.tran.or.non-trans}, hence it remains to show that it has a dense set of periodic points.  Denote by $Per(Q_K)$ the set of periodic points of $Q_K$.
    
    We would like to recall that $ Q_K^n \delta_t=e^{int}\delta_t$, for any $n\in \Z_+$ and $t\in K$, the details can be found in the proof of Proposition 3.9 \cite{BMPS}.
    
    Consider the set $A=\Big\{\sum_{n=1}^k 2\pi \epsilon_n/2^{6^n}: \epsilon\in\{0, 1\}^{\{1, \dots, k\}}, k\in \N\Big\}$.
    
   Observe that $\sum_{n=1}^k 2\pi \epsilon_n/2^{6^n}=2\pi m/2^{6^k}$, for some $m$ and any $\epsilon\in\{0, 1\}^{\{1, \dots, k\}}$. So, clearly $\{\delta_t: t\in A\}\subseteq Per(Q_K)$. Moreover, if $r_1=2\pi m_1/2^{6^{n_1}}\in A, r_2=2\pi m_2/2^{6^{n_2}}\in A$, then $Q_K^{2^{6^{n_1}}2^{6^{n_2}}}(\alpha_1\delta_{r_1}+\alpha_2\delta_{r_2})=\alpha_1\delta_{r_1}+\alpha_2\delta_{r_2}$ for any $\alpha_1, \alpha_2 \in \C$, so $span\{\delta_t: t\in A\}\subseteq Per(Q_K)$.
    
    On the other hand, since $A$ is dense in $K$, we deduce that $\overline{\{\delta_t: t\in A\}}=\{\delta_t: t\in K\}$. Indeed, for any $r\in K$ there exists a sequence $(r_n)_n\subseteq A$ such that $r_n$ tends to $r$. Hence, $\norm{\delta_r-\delta_{r_n}}=\sup_{\norm{f}=1}\abs{f(r)-f(r_n)}$ tends to $0$.
    
    Hence, $X_K=\overline{span\{\delta_t: t\in K\}}=\overline{span\overline{\{\delta_t: t\in A\}}}=\overline{span\{\delta_t: t\in A\}}\subseteq \overline{Per(Q_K)}$. So, $Per(Q_K)$ is dense in $X_K$.
     \end{proof}
      The set $A\subset \N$ is \emph{thick} if $A$ contains arbitrarily long intervals, i.e. for every $L>0$ there exists $n\geq 1$ such that $\{n, n+1, \dots, n+L\}\subset A$.
       
    Now, in order to obtain a mixing operator $T$ such that $(T, T^2)$ is not $d$-syndetic, it will be enough to show that the sequence of operators $\big(2Q_K^{a_n}-Q_K^{2a_n}\big)_n$ is non-transitive along a thick set $A=(a_n)$. 
    We have the following result.
    
   \begin{proposition}
   \label{seq.non.univ}
   Let $K$ be the compact subset of $[-\pi, \pi]$ defined in (\ref{def.set.K}), then the sequence $\big(2Q_K^{k_{n,r}}-Q_K^{2k_{n,r}}\big)_{n\in\N, 0\leq r \leq n}$ of continuous linear operators on $X_K$ is non-transitive, where $k_{n, r}=2^{6^n}-r$ with $0\leq r\leq n, n\in \N$.
   \end{proposition}
   Now we are in position to prove Theorem \ref{op.mixing.tuple.not.d-S}.
 
 We adopt the same sketch of proof of Theorem 3.8 \cite{BMPS}, still we expose here all the details.
 We need to show a mixing and chaotic operator $T$ such that $(T, T^2)$ is not $d$-syndetic. 
 
   Let $K$ be the compact set defined in (\ref{def.set.K}). By Proposition \ref{Q_K.tran.or.non-trans} and Lemma \ref{Q_K_chaotic}, $Q_K$ is mixing and chaotic operator on the separable infinite dimensional Hilbert space $X_K$. On the other hand, by Proposition \ref{seq.non.univ}, $\big(2Q_K^{a_n}-Q_K^{2a_n}\big)_{n\in \N}$ is non-transitive for some thick set $A$ written increasingly as $A=(a_n)_n$. Hence, there exists non-empty open sets $U, V$ in $X_K$ such that $\big(2Q_K^{a_n}-Q_K^{2a_n}\big)(U)\cap V=\emptyset$, for any $n\in \N$. In other words, 
   \[
   \big\{n\in \N: (2Q_K^n-Q_K^{2n})(U)\cap V\neq\emptyset\big\}\cap A=\emptyset,
   \]
    i.e. the set $\big\{n\in \N: (2Q_K^n-Q_K^{2n})(U)\cap V\neq\emptyset\big\}$ cannot be syndetic. In particular, $(Q_K, Q_K^2)$ is not $d$-syndetic. Indeed, pick a non-empty open set $V_0$ such that $2V_0-V_0\subseteq V$(denote $B(x;r)$ the open ball centered at $x$ in $X_K$ with radius r. Pick $x\in X_K, r\in \R_+$ such that $B(x;r)\subset V$, then set $V_0:=B(x;r/3)$). 
    Hence,
    \[
    \big\{n\in \N: U\cap Q_K^{-n}(V_0)\cap Q_K^{-2n}(V_0)\neq\emptyset\big\}\subseteq \big\{n\in \N: (2Q_K^n-Q_K^{2n})(U)\cap V\neq\emptyset\big\}.
    \]
  Consequently, $\big\{n\in \N: U\cap Q_K^{-n}(V_0)\cap Q_K^{-2n}(V_0)\neq\emptyset\big\}$ cannot be a syndetic set and then $(Q_K, Q_K^2)$ is not $d$-syndetic. Since all separable infinite dimensional Hilbert spaces are isomorphic to $l^2$, there is a mixing and chaotic $T\in \Lin(l^2)$ such that $(T, T^2)$ is not $d$-syndetic. This concludes the proof of Theorem \ref{op.mixing.tuple.not.d-S}.   
   
   In order to close this subsection, we need to prove Proposition \ref{seq.non.univ}, which follows exactly the same sketch of proof of Proposition 3.10 \cite{BMPS}, except that instead of using Lemma A.3 \cite{BMPS} as the authors of \cite{BMPS} did, we just use Lemma \ref{convenient-lema_A.3} (below) in an analogous way. So it suffices to give the proof of Lemma \ref{convenient-lema_A.3}. 
   
   Now, in order to prove Lemma \ref{convenient-lema_A.3} we need to quote another two lemmas proved in \cite{BMPS} that we state without proof.
   \begin{lemma} Lemma A.1 \cite{BMPS}
   
   \label{Lemma_A.1_BMPS} 
   Let $f\in W^{2,2}[-\pi, \pi], f(-\pi)=f(\pi), f'(-\pi)=f'(\pi), c_0=\norm{f}_{L^{\infty}[-\pi, \pi]}$ and $c_1=\norm{f''}_{L^2[-\pi, \pi]}$. Then $\norm{f}_{W^2[-\pi, \pi]}\leq \sqrt{3c_1^2+c_0^2}$.
   \end{lemma}
   
   \begin{lemma} Lemma A.2 \cite{BMPS}
   
   \label{Lemma_A.2_BMPS}
   Let $-\infty <\alpha<\beta<\infty$ and $a_0, a_1, b_0, b_1\in \C$. Then there exists $f\in C^2[\alpha, \beta]$ such that
   \[
   f(\alpha)=a_0, \quad f'(\alpha)=a_1, \quad f(\beta)=b_0, \quad f'(\beta)=b_1,
   \]
   \[
   \norm{f}_{L^{\infty}[\alpha, \beta]}\leq \abs{a_0+b_0}/2+\abs{a_0-b_0}/2+(\beta-\alpha)(\abs{a_1}+\abs{b_1})/5,
   \]
   \[
   \norm{f''}^2_{L^2[\alpha, \beta]}\leq \frac{24\abs{a_0-b_0}^2}{(\beta-\alpha)^3}+\frac{12\cdot(\abs{a_1}^2+\abs{b_1}^2)}{\beta-\alpha}.
   \]
   \end{lemma}
   
    \begin{lemma}
    \label{convenient-lema_A.3}
     There exists a sequence $(f_{2^{6^n}-r})_{n\in \N, 0\leq r\leq n}$ of $2\pi$-periodic functions on $\R$ such that $\left.f_{2^{6^n}-r}\right|_{[-\pi, \pi]}\in W^{2, 2}[-\pi, \pi]$, the sequence $\big(\norm{f_{2^{6^n}-r}}_{W^{2, 2}[-\pi, \pi]}\big)_{n, r}$ is bounded and $f_{2^{6^n}-r}(x)=2 e^{i ({2^{6^n}-r}) x}-e^{2i({2^{6^n}-r})x}$ whenever $\abs{x-\frac{2\pi m}{2^{6^n}}}\leq 2/(2^{6^n})^5$, for some $m\in \Z$ and every $n\in \N, 0\leq r \leq n$.
    \end{lemma}
    
    \begin{proof}
    We obtain the proof of this lemma doing convenient slight modifications in the proof of Lemma A.3 \cite{BMPS}.
    
    For $n\in \N, 0\leq r \leq n$, let $k_{n, r}=2^{6^n}-r$ and $h_{k_{n,r}}=2e^{ik_{n,r}x}-e^{2i k_{n, r}x}$. Note that $h_{k_{n,r}}$ is periodic with period $2\pi/k_{n, r}$. Let also $\alpha_{n,r}=\frac{2}{(2^{6^n})^5}-\frac{2\pi}{2^{6^n}}$ and $\beta_{n,r}=-2/(2^{6^n})^5$. By Lemma \ref{Lemma_A.2_BMPS}, there is $g_{k_{n,r}}\in C^2[\alpha_{n,r}, \beta_{n,r}]$ such that
    \[
     g_{k_{n,r}}(\alpha_{n,r})=h_{k_{n,r}}(2/(2^{6^n})^5), \qquad g_{k_{n,r}}(\beta_{n,r})=h_{k_{n,r}}(-2/(2^{6^n})^5),
     \]
    \begin{equation}
    \label{values_of_g_k_n,r}
    g'_{k_{n,r}}(\alpha_{n,r})=h'_{k_{n,r}}(2/(2^{6^n})^5), \qquad g'_{k_{n,r}}(\beta_{n,r})=h'_{k_{n,r}}(-2/(2^{6^n})^5),
    \end{equation}
    \[
    \norm{g_{k_{n,r}}}_{L^\infty_{[\alpha_{n,r}, \beta_{n,r}]}}\leq \max\{\abs{h_{k_{n,r}}(2/(2^{6^n})^5)}, \abs{h_{k_{n,r}}(-2/(2^{6^n})^5)}\}
    \]
    \begin{equation}
    \label{estimate_L_inf}
+\frac{(\beta_{n,r}-\alpha_{n,r})}{5}\big(\abs{h'_{k_{n,r}}(2/(2^{6^n})^5)}+\abs{h'_{k_{n,r}}(-2/(2^{6^n})^5)}\big),
    \end{equation}
    \[
    \norm{g''_{k_{n,r}}}^2_{L^2_{[\alpha_{n,r}, \beta_{n,r}]}}\leq  \frac{24 \abs{h_{k_{n,r}}(2/(2^{6^n})^5)-h_{k_{n,r}}(-2/(2^{6^n})^5)}^2}{(\beta_{n,r}-\alpha_{n,r})^3}
    \]
    \begin{equation}
    \label{estimate_L_2}
     + 12 \frac{\abs{h'_{k_{n,r}}(2/(2^{6^n})^5)}^2+\abs{h'_{k_{n,r}}(-2/(2^{6^n})^5)}^2}{(\beta_{n,r}-\alpha_{n,r})}.
    \end{equation}
    
    The equalities (\ref{values_of_g_k_n,r}) imply that there is a unique $f_{k_{n, r}}\in C^1(\R)$ such that $f_{k_{n, r}}$ is periodic with period $2\pi/2^{6^n}$,  $\left.f_{k_{n, r}}\right|_{[\alpha_{n, r}, \beta_{n,r}]}=g_{k_{n,r}}$ and $\left.f_{k_{n, r}}\right|_{[\beta_{n, r}, \alpha_{n,r} +2\pi/2^{6^n}]}=h_{k_{n,r}}$.
    
    Periodicity of $f_{k_{n, r}}$ with period $2\pi/2^{6^n}$ and the equality $\left.f_{k_{n, r}}\right|_{[\beta_{n, r}, \alpha_{n,r} +2\pi/2^{6^n}]}=h_{k_{n,r}}$ imply that $f_{k_{n,r}}(x)=2 e^{i ({2^{6^n}-r}) x}-e^{2i({2^{6^n}-r})x}$ whenever $\abs{x-\frac{2\pi m}{2^{6^n}}}\leq 2/(2^{6^n})^5$, for every $m\in \Z$ with $\abs{2m}\leq 2^{6^n}$ and every $n\in \N, 0\leq r \leq n$. Since $f_{k_{n, r}}$ is piecewise $C^2, \left.f_{k_{n, r}}\right|_{[-\pi, \pi]}\in W^{2, 2}[-\pi, \pi]$. It remains to verify that the sequence $\big(\norm{f_{k_{n, r}}}_{W^{2, 2}[-\pi, \pi]}\big)_{n, r}$ is bounded.

    Using the inequality $\abs{e^{it}-e^{is}}\leq \abs{t-s}$ for $t, s \in \R$, we have 
    \[
    \abs{h'_{k_{n,r}}(2/(2^{6^n})^5)}=\abs{h'_{k_{n,r}}(-2/(2^{6^n})^5)}\leq 2 (2^{6^n}-r)^2 \cdot 2/(2^{6^n})^5.
    \]
    Hence by (\ref{estimate_L_inf}),
    \[
     \norm{f_{k_{n,r}}}_{L^\infty_{[\alpha_{n,r}, \beta_{n,r}]}}\leq 3+ 5^{-1}(\frac{2\pi}{2^{6^n}}-\frac{4}{(2^{6^n})^5})\cdot 8 \frac{(2^{6^n}-r)^2}{ (2^{6^n})^5}<9.
    \]
    Since $\norm{h_{k_{n,r}}}_{L^\infty_{[\beta_{n,r}, \alpha_{n,r}+2\pi/2^{6^n}]}}\leq 3$ and $f_{k_{n,r}}$ is $2\pi/2^{6^n}$-periodic, we obtain,
    \begin{equation}
    \label{est_L_inf}
    \norm{f_{k_{n, r}}}_{L^{\infty}_{[-\pi, \pi]}}\leq \max\{3, 9\}=9.
    \end{equation}
    
    Next, 
    \[
    \abs{h_{k_{n,r}}(2/(2^{6^n})^5)-h_{k_{n,r}}(-2/(2^{6^n})^5)}=
    \]
    \[
     \abs{2\Big(e^{i(2^{6^n}-r)\frac{2}{(2^{6^n})^5}}- e^{i(2^{6^n}-r)\frac{(-2)}{(2^{6^n})^5}}\Big)- \Big(e^{2i(2^{6^n}-r)\frac{2}{(2^{6^n})^5}}- e^{2i(2^{6^n}-r)\frac{(-2)}{(2^{6^n})^5}}\Big)}=
    \]
    \[
    \abs{4 \sin{\Big(2 \cdot \frac{(2^{6^n}-r)}{(2^{6^n})^5}\Big)}-2 \sin{\Big(4 \cdot \frac{(2^{6^n}-r)}{(2^{6^n})^5}\Big)}}=
    \]
    \[
    4 \sin{\Big(2 \cdot \frac{(2^{6^n}-r)}{(2^{6^n})^5}\Big)}\Big(1-\cos{\big(2 \cdot \frac{(2^{6^n}-r)}{(2^{6^n})^5}\big)}\Big)=
    \]
    \[
    16 \sin^3{\Big(\frac{(2^{6^n}-r)}{(2^{6^n})^5}\Big)}\cos\Big(\frac{(2^{6^n}-r)}{(2^{6^n})^5}\Big)\leq
    \]
    \[
    16 \Big(\frac{(2^{6^n}-r)}{(2^{6^n})^5}\Big)^3\leq \frac{16}{(2^{6^n})^{12}}.
    \]
    
    On the other hand, 
    \[
    \frac{\abs{h'_{k_{n,r}}(2/(2^{6^n})^5)}^2+\abs{h'_{k_{n,r}}(-2/(2^{6^n})^5)}^2}{(\beta_{n,r}-\alpha_{n,r})}\leq \frac{32 \frac{(2^{6^n}-r)^4}{(2^{6^n})^{10}}}{\frac{2\pi}{2^{6^n}}-\frac{4}{(2^{6^n})^5}}\leq
    \]
    \[
     \frac{\frac{32}{(2^{6^n})^6}}{\frac{2\pi}{2^{6^n}}-\frac{4}{(2^{6^n})^5}}=\frac{32}{2\pi(2^{6^n})^5-4\cdot 2^{6^n}}\leq \frac{32}{2\pi(2^{6^n})^5-4 (2^{6^n})^5}\leq \frac{16}{(2^{6^n})^5}.
    \]
    Hence by (\ref{estimate_L_2}),
    \[
     \norm{f''_{k_{n,r}}}^2_{L^2_{[\alpha_{n,r}, \beta_{n,r}]}}\leq 24\cdot \frac{(\frac{16}{(2^{6^n})^{12}})^2}{\Big(\frac{2\pi}{2^{6^n}}-\frac{4}{(2^{6^n})^5}\Big)^3}+12 \cdot \frac{16}{(2^{6^n})^5}\leq 
    \]
    \[
    24\cdot \frac{16^2 \cdot(2^{6^n})^{-24}}{\Big(\frac{2\pi}{2^{6^n}}-\frac{4}{2^{6^n}}\Big)^3}+12 \cdot \frac{16}{(2^{6^n})^5}\leq 
    \]
    \[
    \frac{24 \cdot 16^2}{8\cdot (2^{6^n})^{21}}+ \frac{12 \cdot 16}{(2^{6^n})^5}\leq \frac{960}{(2^{6^n})^5}.
    \]
    
    Since $\abs{h''_{k_{n,r}}(x)}\leq 6 (k_{n,r})^2$ for $x\in [\beta_{n, r}, \alpha_{n, r} +2\pi/2^{6^n}]$ we have,
    \[
     \norm{f''_{k_{n,r}}}^2_{L^2_{[\beta_{n, r}, \alpha_{n, r} +2\pi/2^{6^n}]}}\leq 36 \cdot (2^{6^n}-r)^4\cdot \frac{4}{(2^{6^n})^5}\leq \frac{144}{2^{6^n}}.
    \]
    
    Hence, 
    \[
    \norm{f''_{k_{n,r}}}^2_{L^2_{[\alpha_{n, r}, \alpha_{n, r} +2\pi/2^{6^n}]}}\leq \frac{960}{(2^{6^n})^5}+\frac{144}{2^{6^n}}\leq \frac{1104}{2^{6^n}}.
    \]
    Since $f''_{k_{n,r}}$ is $2\pi/2^{6^n}$-periodic then  
    \begin{equation}
\label{est_L_2}    
    \norm{f''_{k_{n,r}}}^2_{L^2_{[-\pi, \pi]}}=2^{6^n}\cdot \norm{f''_{k_{n,r}}}^2_{L^2_{[\alpha_{n, r}, \alpha_{n, r} +2\pi/2^{6^n}]}}\leq 1104.
    \end{equation}
     Now, by Lemma \ref{Lemma_A.1_BMPS} and using (\ref{est_L_2}) and (\ref{est_L_inf}) we obtain
    \[
    \norm{f_{k_{n,r}}}_{W^{2,2}[-\pi, \pi]}\leq \sqrt{3\cdot 1104+9^2}<64,
    \]
    for each $n\in \N, 0\leq r \leq n$.
    \end{proof}

\subsection{Proof of Theorem \ref{generalcase.rhc.implies.d-synd}}
 The main ingredient of the proof of Theorem \ref{generalcase.rhc.implies.d-synd} is a result due to Bergelson and McCutcheon concerning essential idempotents of $\beta\N$ (the Stone-\v{C}ech compactification of $\N$), and Szemer\'{e}di's theorem for generalized polynomials \cite{BeMc}. So, we need first some background on $\beta\N$.
  
 Recall that a \emph{filter} is a family that is invariant by finite intersections, i.e. $\F$ is a family such that for any $A\in \F, B\in \F$ implies $A\cap B\in \F$.  The collection of all maximal filters (in the sense of inclusion) is denoted by $\beta \N$. Elements of $\beta \N$ are known as \emph{ultrafilters}; endowed with an appropiate topology, $\beta \N$ becomes the Stone-\v{C}ech Compactification of $\N$. Each point $i\in \N$ is identified with a principal ultrafilter $\mathfrak{U}_i:=\{A\subseteq \N:i\in A\}$ in order to obtain an embedding of $\N$ into $\beta \N$. For any $A\subseteq \N$ and $p\in\beta \N$, the closure of $A$, $cl A$ in $\beta \N$ is defined as follows, $p\in clA$ if and only if $A\in p$. Given $p, q\in\beta \N$ and $A\subseteq \N$, the operation $(\N, +)$ can be extended to $\beta \N$ in such a way as to make $(\beta \N, +)$ a compact right topological semigroup. The extended operation can be defined as $A\in p + q$ if and only if $\{n\in \N: -n+A\in q\}\in p$. Now, according to a famous theorem of Ellis, idempotents (with respect to $+$) exist. Let $E(\N)=\{p\in \beta \N: p=p + p\}$ be the collection of idempotents in $\beta \N$. For further details see \cite{HiSt}. Given a family $\mathscr{F}$, the \emph{dual} family $\mathscr{F}^*$ consists of all sets $A$ such that $A\cap F\neq \emptyset$, for every $F\in \mathscr{F}$.
 The following is a well-known result.
 \begin{lemma}
 \label{LemmaBeDo}
 $(1)$ If $\F$ is an ultrafilter, then $\F^*=\F$.
 
 $(2)$ If $\F=\cup_\alpha \F_\alpha$, then $\F^*=\cap_\alpha \F^*_\alpha$.
 
 In particular, whenever $\F$ is a union of some collection of ultrafilters, then $\F^*$ is the intersection of the same collection of ultrafilters.
 \end{lemma}

    The collection of essential idempotents is commonly referred to in the literature as $\mathcal{D}$. 

  The collection $\E$ (of $D$-sets) is the union of all idempotents $p\in \beta\N$ such that every member of $p$ has positive upper Banach density. Accordingly, $\E^*$ is the intersection of all such idempotents. 

   The following is a result of ergodic Ramsey theory due to Bergelson and McCutcheon \cite{BeMc}. It is indeed a sort of Szemer\'{e}di's theorem stated originally for generalized polynomials and it will be crucial for proving Theorem \ref{generalcase.rhc.implies.d-synd}.
   \begin{theorem} Theorem 1.25 \cite{BeMc}
   \label{theor.BeMc}
     Let $F\subset \N$ have positive upper Banach density and $g_1\dots, g_r$ be polynomials, then
   \[
   \Big\{k\in\N: \overline{Bd}\big(F\cap (F-g_1(k))\cap \dots \cap (F-g_r(k))\big)>0\Big\}\in \E^*.
   \]
   \end{theorem}

  We can now prove Theorem \ref{generalcase.rhc.implies.d-synd}.

  Fix $r\in \N$. Let $T$ be reiteratively hypercyclic, then there exists $x\in X$ such that $\overline{Bd}\big(N(x, U)\big)>0$, for any non-empty open set $U$ in $X$. First, let us see that 
   \begin{equation}
   \label{Tis.centralG}
   N_T( \underbrace{U, \dots, U}_{r};U)=\{k\geq 0: T^{-k}U\cap\dots \cap T^{-rk}U\cap U\neq \emptyset\}\in \E^*
   \end{equation}
    for any non-empty open set $U$ in $X$. Let $U$ non-empty open set, then 
\begin{align*}
   A_U:=\big\{k\geq 0: \overline{Bd}\Big(N(x, U)\cap \big(N(x, U)-k\big)\cap\dots \cap \big(N(x, U)-rk\big)\Big)>0\big\}\\
   \subseteq \big\{k\geq 0: T^{-k}U\cap \dots \cap T^{-rk}U\cap U \neq \emptyset\big\}.
 \end{align*}
In fact, let $k\in A_U$, then there exists a set $A$ with positive upper Banach density such that for any $n\in A$ it holds $T^{n+ik}x\in U$, for any $i\in\{0,\dots, r\}$. Consequently, $T^nx\in T^{-k}U\cap \dots \cap T^{-rk}U\cap U$. Now, by Theorem \ref{theor.BeMc}, it follows that $A_U\in \E^*$. Thus condition (\ref{Tis.centralG}) holds.

 Next, let $(U_j)_{j=0}^r$ be a finite sequence of non-empty open sets in $X$. Now, suppose that $(T, \dots, T^r)$ is $d$-transitive, we must show that $N_T(U_1, \dots, U_r; U_0)$ is a syndetic set. In fact, there exists $n\in \N$ such that 
 \[
 V_n:=T^{-n}U_1\cap\dots \cap T^{-rn}U_r\cap U_0\neq \emptyset.
 \]
 Thus $V_n$ is open, then pick $O_1, O_2$ non-empty open sets such that $O_1+O_2\subset V_n$, then 
 \begin{equation}
 \label{incl.O1.O2}
 T^{jn}(O_1+O_2)\subset U_j, \qquad \mbox{for any $j\in \{0, \dots, r\}$}.
 \end{equation}
 It is known that $\E^*$ is a filter. Now, by (\ref{Tis.centralG}) we have
 \[
 A:=N_T( \underbrace{O_1, \dots, O_1}_{r};O_1) \cap N_T( \underbrace{O_2, \dots, O_2}_{r};O_2)\in \E^*.
 \]
 In addition, it is well known that  each set in $\E^*$ is indeed syndetic \cite{BeDo}. Hence, $A$ is syndetic. Let us show that $A+n\subseteq N(U_1, \dots, U_r; U_0)$, then we are done because $A+n$ is syndetic, since the collection of syndetic sets is shift invariant.
 
 In fact, let $t\in A+n$, then $t-n\in A$, which means 
   \begin{align*}
   T^{-t}T^n(O_1)\cap\dots \cap T^{-rt}T^{rn}(O_1)\cap O_1\neq \emptyset\\
    T^{-t}T^n(O_2)\cap\dots \cap T^{-rt}T^{rn}(O_2)\cap O_2\neq \emptyset.
   \end{align*}
 By the linearity of $T$ we obtain
 \[
  T^{-t}\big(T^n(O_1+O_2)\big)\cap\dots \cap T^{-rt}\big(T^{rn}(O_1+O_2)\big)\cap (O_1+O_2)\neq \emptyset.
 \]
 Then we conclude by (\ref{incl.O1.O2}), i.e.
 \[
 T^{-t}U_1\cap\dots \cap T^{-rt}U_r\cap U_0\neq \emptyset.
\]
This concludes the proof of Theorem \ref{generalcase.rhc.implies.d-synd}.

\section{Tuple of powers of a weighted shift}
 In linear dynamics recurrence properties are frequently studied first in the context of weighted backward shifts. 

 Each bilateral bounded weight $w=(w_k)_{k\in \Z}$, induces a \emph{bilateral weighted backward shift} $B_w$ on $X=c_0(\Z)$ or $l^p(\Z) (1\leq p<\infty)$, given by $B_{w}e_k:=w_{k}e_{k-1}$, where $(e_k)_{k\in \Z}$ denotes the canonical basis of $X$. 
 
 Analogously,  each unilateral bounded weight $w=(w_n)_{n\in \Z_+}$ induces a \emph{unilateral weighted backward shift} $B_w$ on $X=c_0(\Z_+)$ or $l^p(\Z_+) (1\leq p<\infty)$, given by $B_{w}e_n:=w_{n}e_{n-1}, n\geq 1$ with $B_{w}e_0:=0$, where $(e_n)_{n\in \Z_+}$ denotes the canonical basis of $X$. 

 As previously mentioned, the authors of \cite{BMPS} proved that for any weighted shift $B_w$, the following holds: $B_w$ is mixing if and only if $(B_w, \dots, B_w^r)$ is $d$-mixing for any $r\in \N$. The aim of this section is to show that this result extends to those families on $\N$ frequently studied in Ramsey theory.
 
 Let us summarize some families commonly used in Ramsey theory. 
     \begin{itemize}
\item $\I=\{A \subseteq \N: A \mbox{ is infinite}\}$;
\item $\Delta=\{A\subseteq \N: B-B\subseteq A, \mbox{for some infinite set }B\}$;
\item $\IP=\{A\subseteq \N: \exists (x_n)_n\subseteq \N, \sum_{n\in F}x_n\in A, \mbox{for any finite set }F\}$;
\item The set $A$ is \emph{piecewise syndetic} ($A\in \PS$ for short) if $A$ can be written as the intersection of a thick and a syndetic set.
\end{itemize}
  It is known that  $\I^*$(family of cofinite sets), $\Delta^*, \IP^*$ and $\PS^*$ are filters. In addition, $\I^*\subsetneqq \Delta^* \subsetneqq \IP^* \subsetneqq \Synd$ and $\I^*\subsetneqq \PS^* \subsetneqq \Synd$, where $\Synd$ denote the family of syndetic sets. For a rich source on this subject we refer the reader to \cite{HiSt}.  
 
 The main result of this section is the following.
 
\begin{theorem}
\label{tupleofpower.filter.bilateral}
 Let $\F$ be the family $\Delta^*, \IP^*, \PS^*$ or $\Synd$ then for any $r\in \N$ the following are equivalent:

(i) $T$ is an $\F$-operator;

(ii) $T\oplus\dots\oplus T^r$ is an $\F$-operator on $X^r$.

 In particular, a bilateral (unilateral) weighted backward shift $B_w$ on $c_0$ or $l^p (1\leq p< \infty)$ is an $\F$-operator if and only if $(B_w, \dots, B_w^r)$ is d-$\F$. 
\end{theorem}
 \begin{remark}
 Obviously, mixing operators are $\Delta^*$-operators, but the converse is not true as exhibited in \cite{BMPP} and the example is a weighted shift. Therefore, the conclusion of Theorem \ref{tupleofpower.filter.bilateral} concerning weighted shifts does not necessarily follows from the statement: $B_w$ is mixing if and only if $(B_w, \dots, B_w^r)$ is $d$-mixing, for any $r\in \N$, shown in \cite{BMPS}.
  \end{remark}
  
In order to prove Theorem \ref{tupleofpower.filter.bilateral} we will need the following results.
 
 Recall that any tuple of powers of a fixed backward weighted shift on $c_0$ or $l^p$ is $d$-transitive if and only if it is $d$-hypercyclic. This follows by Theorem 2.7 \cite{BePe} and Theorem 4.1 \cite{BePe}. Now, combining Theorem 4.1 \cite{BePe} and Theorem 2.5 \cite{Sa} in its bilateral (unilateral) version, we obtain the following two propositions below.
\begin{proposition}
\label{cor.charact.bilateral}
Let $X=c_0(\Z)$ or $l^p(\Z) (1\leq p < \infty), w=(w_j)_{j\in \Z}$ a bounded bilateral weight sequence, $\F$ a filter on $\N$ and $r_0=0<1\leq r_1<\dots < r_N$, then the following are equivalent: 

(i) $(B_w^{r_1}, \dots, B_w^{r_N})$ is d-$\F$,

(ii) $\oplus_{0\leq s < l \leq N}B_w^{(r_l-r_s)}$ is $\F$-operator on $X^{\frac{N(N+1)}{2}}$,

(iii) for any $M>0, j\in \Z$ and $0\leq s<l\leq N$ it holds
\[
\Big\{m\in \N: \prod_{i=j+1}^{j+m(r_l-r_s)}\abs{w_i}>M\Big\}\in \F,
\]
\[
\Big\{m\in \N: \frac{1}{\prod_{i=j-m(r_l-r_s)+1}^j\abs{w_i}}>M\Big\}\in \F.
\]
\end{proposition}

\begin{proposition}
\label{cor.charact.unilateral}
Let $X=c_0(\Z_+)$ or $l^p(\Z_+) (1\leq p < \infty),w=(w_n)_{n\in \Z_+}$ a bounded unilateral weight sequence, $\F$ a filter on $\N$ and $r_0=0<1\leq r_1<\dots < r_N$, then the following are equivalent: 

(i) $(B_w^{r_1}, \dots, B_w^{r_N})$ is d-$\F$,

(ii) $\oplus_{0\leq s < l \leq N}B_w^{(r_l-r_s)}$ is $\F$-operator on $X^{\frac{N(N+1)}{2}}$,

(iii) for any $M>0, j\in \Z_+$ and $0\leq s<l\leq N$ it holds
\[
\Big\{m\in \N: \prod_{i=j+1}^{j+m(r_l-r_s)}\abs{w_i}>M\Big\}\in \F.
\]

\end{proposition}
The following results of Ramsey theory concern the preservation of certain notions of largeness in products.
 \begin{proposition} Corollary 2.3 \cite{BeHi}
   \label{corol2.3BeHi}
   Let $l\in \N$ and $I$ be a subsemigroup of $\N^l$,
   
   a) if $B$ is an $IP^*$ set in $\N$, then $B^l\cap I$ is an $IP^*$ set in $I$
   
   b) if $B$ is an $\Delta^*$ set in $\N$, then $B^l\cap I$ is an $\Delta^*$ set in $I$.
   \end{proposition}
   
    \begin{proposition} Corollary 2.7 \cite{BeHi}
    \label{corol2.7BeHi}
   Let $l\in \N$ and $I$ be a subsemigroup of $\N^l$,
   
   a) if $B$ is an $PS^*$ set in $\N$, then $B^l\cap I$ is an $PS^*$ set in $I$.
   \end{proposition}
We are now finally able to prove Theorem \ref{tupleofpower.filter.bilateral}.

\textbf{Proof of Theorem \ref{tupleofpower.filter.bilateral}.}

 If $T\oplus \dots \oplus T^r$ is $\F$-operator on $X^r$ for some $r\in \N$, obviously $T$ is $\F$-operator. Conversely, let $T$ an $\F$-operator,  $r\in \N$ and $U, V$ non-empty open sets, we need to show that $N(U, V)\in t\F$, for any $t=1,\dots, r$.
   
   Denote,
\begin{equation}
\label{Bergelson.cond} 
A=\{m, 2m, \dots, rm: m\in \N\}\cap (\underbrace{N(U, V)\times \dots  \times N(U, V)}_{r-times}).
\end{equation}
   By Proposition \ref{corol2.3BeHi}, we have that if $N(U, V)$ is $IP^*$-set ($\Delta^*$-set) in $\N$, then $A$ is $IP^*$-set ($\Delta^*$-set) in $\{m, 2m, \dots, rm: m\in \N\}$. Analogously, by Proposition \ref{corol2.7BeHi}, we have that if $N(U, V)$ is $PS^*$-set in $\N$, then $A$ is $PS^*$-set  in $\{m, 2m, \dots, rm: m\in \N\}$.

 Denote $\prod_i$ the projection onto the $i$-th coordinate. It is not difficult to see that $\prod_1(A)\in \F$, for $\F=\Delta^*, \IP^*, \PS^*$.  
 Then, (\ref{Bergelson.cond}) is equivalent to say 
 \[
B=\{m\in \N: tm\in N(U, V)\}\in \F,
 \]
for any $t=1,\dots, r$.

Hence, $tB\subseteq N(U, V)$ and $B\in \F$. Then $N(U, V)\in t\F$ for any $t=1,\dots, r$. Since $\F=\Delta^*, \IP^*, \PS^*$; it is a filter, then it is not difficult to see that $T\oplus \dots \oplus T^r$ is indeed an $\F$-operator on $X^r$. 

If $B_w$ is a weighted shift on $c_0$ or $l^p$ and $\F=\Delta^*, \IP^*, \PS^*$; by Proposition \ref{cor.charact.bilateral} (Proposition \ref{cor.charact.unilateral}), we have $B_w$ is an $\F$-operator if and only if $(B_w, \dots, B_w^r)$ is d-$\F$ for any $r\in \N$. 

Finally, let $\F$ be the family of syndetic sets. Just recall that $T$ is syndetic operator if and only if $T$ is $\PS^*$-operator \cite{BMPP}. Hence $T$ is syndetic operator  if and only if $T\oplus\dots\oplus T^r$ is $\PS^*$-operator on $X^r$ for any $r\in \N$. If $B_w$ is a weighted shift then $B_w$ is syndetic operator if and only if $(B_w, \dots, B_w^r)$ is $d$-$\PS^*$ for any $r\in \N$. This concludes the proof of Theorem \ref{tupleofpower.filter.bilateral}. \ \ 
    

\end{document}